\documentclass[letterpaper, 10 pt, conference]{ieeeconf}  

\IEEEoverridecommandlockouts                              
\usepackage{graphics} 
\usepackage{graphicx}
\usepackage{tikz}
\usepackage{subcaption}

\usepackage{amsmath} 
\usepackage{amssymb}  
\usepackage{mathtools}
\usepackage{nicefrac}  

\renewcommand{\theenumi}{(\roman{enumi})}

\newcommand{\oprocendsymbol}{\hbox{$\bullet$}}
\newcommand{\oprocend}{\relax\ifmmode\else\unskip\hfill\fi\oprocendsymbol}

\newcommand{\Ec}{\mathcal{E}}

\newcommand{\Gc}{\mathcal{G}}
\newcommand{\Hc}{\mathcal{H}}

\newcommand{\Uc}{\mathcal{U}}
\newcommand{\Vc}{\mathcal{V}}

\newcommand{\abs}[1]{\left| #1 \right|}

\graphicspath{{epsfiles/}}

\newcommand{\real}{\ensuremath{\mathbb{R}}}

\newcommand{\Torus}{\mathbb{T}}
\newcommand{\lambdamin}{\lambda_{\text{min}}}
\newcommand{\trace}{\text{tr}}
\newcommand{\diag}{\text{diag}}
\newcommand{\Hess}{H}

\newtheorem{theorem}{Theorem}

\newtheorem{lemma}{Lemma}

\newtheorem{example}{Example}
\newtheorem{problem}{Problem}

\newtheorem{conjecture}{Conjecture}

\newenvironment{proofsketch}{\paragraph*{Proof (Sketch)}}{\hfill$\blacksquare$}

\newcommand{\longthmtitle}[1]{\mbox{} \emph{(#1)}}

\usepackage{cite} 

\let\labelindent\relax
\usepackage{enumitem}

\usepackage{xcolor} 
\usepackage{bm}		      		
\usepackage{dsfont}		  		
\newcommand{\N}{\mathbb{N}}		
\newcommand{\Z}{\mathbb{Z}}		

\newcommand*{\QEDA}{\hfill\ensuremath{\triangle}}   %

\title{\LARGE \bf
Global Optimization Through Heterogeneous Oscillator Ising Machines
}

\author{Ahmed Allibhoy, Arthur N. Montanari, Fabio Pasqualetti, Adilson E. Motter
\thanks{This material is based upon work supported in part by ARO award W911NF-24-1-0228.
A. Allibhoy and F. Pasqualetti are with the
Department of Mechanical Engineering at the 
University of California, Riverside, Riverside, CA, 92521, USA. 
Emails: \texttt{\{aallibho, fabiopas\}@ucr.edu}}%
\thanks{A. N. Montanari and A. E. Motter are with the Center for Network Dynamics and 
the Department of Physics and
Astronomy, Northwestern University, Evanston, IL 60208 USA. A. E. Motter is also with the Department of Engineering Sciences \& Applied Mathematics and the Northwestern Institute on Complex Systems, Northwestern University, Evanston, IL 60208, USA. 
Emails: \texttt{\{arthur.montanari, motter\}@northwestern.edu}}
}%

\begin{document}
\maketitle
\thispagestyle{empty}
\pagestyle{empty}

\begin{abstract}
%
%

Oscillator Ising machines (OIMs) are networks of coupled  oscillators that
seek the minimum energy state of an Ising model. Since
many NP-hard problems are equivalent to the minimization of an Ising Hamiltonian, OIMs 
have emerged as a promising
computing paradigm for solving complex optimization  problems 
that are intractable on existing
digital computers. However, their performance is sensitive to the choice of tunable
parameters, and convergence guarantees are mostly lacking.
Here, we 
show that lower energy states 
are more likely to be stable, and that convergence to the global minimizer is often 
improved
by introducing random heterogeneities in the regularization
parameters. Our analysis relates 
the stability properties of Ising configurations to the
spectral properties of a signed graph Laplacian. By examining the spectra of random ensembles of these graphs,
we show that the probability of an equilibrium being asymptotically stable depends
inversely on the value of the Ising Hamiltonian, biasing the system toward low-energy states. Our
numerical results confirm our findings and demonstrate that heterogeneously designed OIMs efficiently converge
to globally optimal solutions with high probability.
\end{abstract}


\vspace{-2.5ex}
\section{Introduction}
The increasing difficulty of scaling up conventional digital computers\textemdash
coupled with the growing relevance of 
combinatorial optimization in artificial intelligence, 
machine learning, and data processing\textemdash has spurred the search for 
novel computing paradigms. One promising alternative is the development of \textit{Ising machines}, 
which are physical systems designed to seek the minimizers 
of the Ising model. The Ising model was originally proposed to 
characterize the magnetic properties of 
materials but is now of interest in 
computing since many NP-hard problems (e.g., graph maximum cut and boolean satisfiability) have been shown to be equivalent to the minimization of an Ising Hamiltonian \cite{AL:14}. As a result, Ising machines can be used as specialized 
hardware 
solvers for many important computing tasks. 

Ising machines may be realized as physical devices using many different 
technologies, including magnetic tunnel junctions, spin-torque arrays,
optoelectronic oscillators, and CMOS circuits. In this paper, 
we focus on \emph{oscillator Ising machines} (OIMs), 
which are implemented using networks of 
coupled phase oscillators. In these machines, the minimizers of the Ising model are encoded as synchronization patterns of the oscillator network \cite{mallick2020using}. 
The dynamics of these systems are related
to the well-studied Kuramoto model \cite{rodrigues2016kuramoto}; however the presence of 
negative edge weights, along with an additional forcing term, 
leads to novel complex dynamical behaviors that are not typically seen 
in classical models of synchronization. The performance of OIMs are sensitive to 
the choice of their regularization parameters. If such parameters are incorrectly tuned, the system may converge to equilibria corresponding to
suboptimal solutions and, as of now, 
there are no formal 
guarantees of convergence to the global minimizers. 

\subsubsection*{Related Work}
Ising machines have received widespread 
interest in recent years;  e.g. see \cite{GC-WP:20, NM-PLM-TB:22} and 
references therein for an overview of 
these devices, applications to computing,
and hardware implementations. 
The realization of Ising machines 
through coupled oscillator networks 
was introduced in~\cite{TW-JR:19}, where it 
was shown that these systems can be interpreted as a gradient flow of an 
objective function 
related to the Ising Hamiltonian. Reference~\cite{YC-MKB-NS-ZL:24} 
classifies all possible equilibria in OIMs 
and their structural stability 
properties. Numerical analysis of the stability of these systems has shown that, for certain graphs, the oscillator parameters can be homogeneously tuned to force the globally optimal solution 
to be stable while most suboptimal solutions are left unstable \cite{MKB-ZL-NS:23}. As a result, these OIMs become biased towards global minimizers. However, to the best of 
our knowledge, convergence guarantees are still lacking, as is a characterization of the relationship between the equilibria of the oscillator network and the corresponding values of the Ising Hamiltonian.
This paper aims to narrow this gap. A key aspect of our technical approach 
is the use of signed graphs and random matrix theory
for the stability analysis of the OIMs. 
Signed graphs have been often used to characterize mutualistic and antagonistic interactions in the study 
of consensus \cite{LP-HS-MM:16, SA-MHDB-MM:17, GS-CA-JSB:19} and synchronization dynamics
\cite{HH-SHS:11-pre, HH-SHS:11-prl}, among others. Complementary approaches in neurocomputation have also investigated the spectra of random network ensembles in order to characterize the multistability of networks with competing equilibria \cite{KR-LFA:06, YH-HS:22}. 

\subsubsection*{Statement of Contributions}
In this paper, we analyze the stability 
properties of OIMs. 
Our theoretical analysis is based on a novel graph-theoretic characterization of the equilibria of OIMs that associates the spin
configurations of the Ising model to signed graphs. This formalism allows us to relate both the energy and
the stability of a given equilibrium to the 
spectral properties of a signed Laplacian matrix. 
We then establish necessary and sufficient conditions for
the stability of the globally optimum equilibria in frustration-free
%
networks. To extend this analysis to networks with frustration--where the signed network structure inherently prevents the simultaneous minimization of all local interaction energies--we turn our attention to ensembles 
of random graphs, in which we derive the conditional moments of the spectrum associated with the system's linearized dynamics, which allows 
us to establish
%
that states corresponding to lower energy values are more likely to be stable. 
Our theoretical analysis and numerical experiments show that
introducing heterogeneity in 
the regularization parameters 
of the oscillators increases 
the spectral variance of 
the Hessian matrix at states corresponding to extreme Hamiltonian values. Thus, suitable heterogeneity 
increases the likelihood of 
stabilizing the global minimizer while
ensuring suboptimal solutions remain unstable.

\section{Preliminaries}

\vspace{-2.5ex}
\subsection{Notation}
Let $\real$, $\real^{N}$, and $\real^{N \times N}$ denote the set of 
real numbers, real vectors with $N$ elements, and real $N \times N$ matrices 
respectively. Let $\real^{N}_{> 0}$ denote 
the set of vectors with positive entries.  
Given a matrix $A \in \real^{N \times N}$, we denote 
its $(i,j)$-th entry by $A_{ij}$.
For a symmetric matrix $Q \in \real^{N \times N}$, 
we denote its smallest eigenvalue 
by $\lambdamin(Q)$, and we write~$Q \succeq 0$ (resp. $Q \succ 0$) 
if 
$Q$ is positive semidefinite (resp. positive definite). 
Let $S^{1} = \real / (2\pi \Z)$ denote the 
$1$-sphere, and $\Torus^{N} = (S^1) \times \cdots \times (S^1)$ the $N$-torus. 

\subsection{The Ising Model}
Here, we review the Ising model. 
Let $G$ be a graph 
with~$N$ nodes and a weighted adjacency 
matrix $J \in \real^{N \times N}$.
For every node $i$ in the network, we 
assign a \emph{spin} 
denoted by $\sigma_i \in \{-1, 1\}$. If $\sigma_i = 1$ (resp. $\sigma_i = -1$), we say that 
node~$i$ is \emph{spin up} (resp. \emph{spin down}). We call the vector of spins 
$\sigma = (\sigma_1, \dots, \sigma_N)\in \{-1, 1\}^N$ a \emph{spin configuration} of 
the network. 
The system energy is a function of the spin configuration, as given 
by the \emph{Ising Hamiltonian}
\begin{equation}
    \label{eq:ising_hamiltonian}
    \mathcal H(\sigma) = -\frac{1}{2}\sum_{i=1}^{N}\sum_{j=1}^{N}J_{ij}\sigma_i\sigma_j.
\end{equation}
Ising machines are physical devices designed to seek configurations in 
which 
the Ising Hamiltonian takes low values, and we are particularly 
interested in spin configurations that globally minimize~\eqref{eq:ising_hamiltonian}. 

The tendency of the spins of two nodes $i$ and $j$ to be aligned or 
antialigned 
depends on the sign of their connecting edge in the 
network. If $J_{ij} > 0$ (resp. $J_{ij} < 0$), the edge 
$(i, j)$  is called \emph{ferromagnetic} (resp. \emph{antiferromagnetic}) 
and the Ising Hamiltonian tends to take lower values when
$\sigma_i = \sigma_j$ (resp. $\sigma_i \neq \sigma_j$) 
since this when $-J_{ij}\sigma_i\sigma_j$, the \emph{local interaction energy} between $i$ 
and $j$, is minimized. An important 
property of many networks is \emph{frustration}, where for 
every spin configuration $\sigma \in \{-1, 1\}^N$ there is 
at least one ferromagnetic edge where $\sigma_i \neq \sigma_j$
or one antiferromagnetic edge where $\sigma_i = \sigma_j$. An example 
of a network that exhibits frustration is shown in Figure \ref{fig:frustration}. 
The presence of frustration in a network makes it difficult 
to find the minimizers of the Ising Hamiltonian, and has important implications for 
the stability analysis of the systems considered in this paper. 

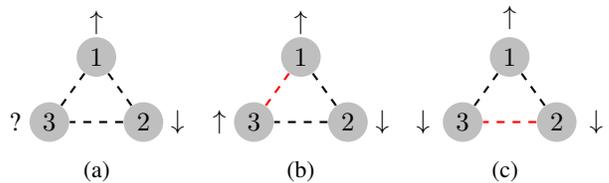
\begin{figure}
    \centering
    \begin{subfigure}{0.3\linewidth}
        \centering
        \begin{tikzpicture}[node distance=13pt]
            \tikzstyle{spinsite}=[circle,fill=black!25,minimum size=15pt,inner sep=0pt, node distance=60pt]
    
            \node (n1) at (0, 0) [spinsite] {$1$};
            \node (n2) at (1.25*0.5, -0.866) [spinsite] {$2$};
            \node (n3) at (-1.25*0.5, -0.866) [spinsite] {$3$};
    
            \draw[thick, dashed] (n1) -- (n2);
            \draw[thick, dashed] (n2) -- (n3);
            \draw[thick, dashed] (n3) -- (n1);
    
            \node[above of=n1] {$\uparrow$};
            \node[right of=n2] {$\downarrow$};
            \node[left of=n3] {?};
        \end{tikzpicture}
        \caption{}
    \end{subfigure}
    \begin{subfigure}{0.3\linewidth}
        \centering
        \begin{tikzpicture}[node distance=13pt]
            \tikzstyle{spinsite}=[circle,fill=black!25,minimum size=15pt,inner sep=0pt, node distance=60pt]
    
            \node (n1) at (0, 0) [spinsite] {$1$};
            \node (n2) at (1.25*0.5, -0.866) [spinsite] {$2$};
            \node (n3) at (-1.25*0.5, -0.866) [spinsite] {$3$};
    
            \draw[thick, dashed] (n1) -- (n2);
            \draw[thick, dashed] (n2) -- (n3);
            \draw[thick, dashed, color=red] (n3) -- (n1);
    
            \node[above of=n1] {$\uparrow$};
            \node[right of=n2] {$\downarrow$};
            \node[left of=n3] {$\uparrow$};
        \end{tikzpicture}
    \caption{}
    \end{subfigure}
    \begin{subfigure}{0.3\linewidth}
        \centering
        \begin{tikzpicture}[node distance=15pt]
            \tikzstyle{spinsite}=[circle,fill=black!25,minimum size=15pt,inner sep=0pt, node distance=60pt]
    
            \node (n1) at (0, 0) [spinsite] {$1$};
            \node (n2) at (1.25*0.5, -0.866) [spinsite] {$2$};
            \node (n3) at (-1.25*0.5, -0.866) [spinsite] {$3$};
    
            \draw[thick, dashed] (n1) -- (n2);
            \draw[thick, dashed, color=red] (n2) -- (n3);
            \draw[thick, dashed] (n3) -- (n1);
    
            \node[above of=n1] {$\uparrow$};
            \node[right of=n2] {$\downarrow$};
            \node[left of=n3] {$\downarrow$};
        \end{tikzpicture}
        \caption{}
    \end{subfigure}
   
    \caption{Example of a frustrated network with antiferromagnetic interactions between nodes (depicted by dashed lines), where for any spin configuration there exists a pair of adjacent nodes with aligned spins. (\textbf{a}) Triangular spin interaction network where spins 1 and 2 are antialigned, thereby minimizing their interaction energy. (\textbf{b}) State in which spin 3 is antialigned with spin 2 and aligned with spin 1; consequently not minimizing the local interaction energy $-J_{13}\sigma_1\sigma_3$. (\textbf{c}) Complementary state in which the local interaction energy between spins 3 and 2, $-J_{23}\sigma_2\sigma_3$,
    is not minimized. }
    \label{fig:frustration}
\end{figure}

\section{Problem Formulation}
Given a graph $G$ with adjacency matrix $J \in \real^{N \times N}$, 
we wish to identify a spin configuration $\sigma \in \{-1, 1\}^{N}$ 
that minimizes the Ising Hamiltonian of the network (i.e. solve $\min_{\sigma} \mathcal H(\sigma)$). 
To find the solution, we use the following continuous-time dynamical system known as an OIM:\begin{equation}
    \label{eq:oim-dynamics}
    \dot{\theta}_i = \sum_{j=1}^{N}J_{ij}\sin(\theta_j - \theta_i) - \mu_i \sin(2\theta_i). 
\end{equation}
Here, $\mu_i \geq 0$ is a tunable parameter. In typical formulations of OIMs,
the parameter $\mu_i$ takes the same value for 
every node (i.e, $\mu_i=\bar{\mu}$ for all $i$).
Here, we allow for heterogeneity 
in these parameters in order to improve convergence to 
global minimizers. 

This system converges to equilibrium points that correspond to spin configurations minimizing \eqref{eq:ising_hamiltonian}. Every spin configuration $\sigma \in \{-1, 1 \}^N$ 
can be associated to an equilibrium\footnote{We occasionally abuse terminology and use the 
term \emph{spin configuration} to refer both to $\sigma \in \{-1, 1 \}^N$ 
and the equilibrium $\theta^*(\sigma) \in \Torus^N$.} $\theta^*(\sigma) \in \Torus^N$:
\begin{equation}
\label{eq:oim-equilibria}
    \theta_i^*(\sigma) = \begin{aligned}
        \begin{cases} 0: \qquad &\sigma_i = 1, \\ 
        \pi: \qquad &\sigma_i = -1.
        \end{cases}
    \end{aligned}  
\end{equation}
The system \eqref{eq:oim-dynamics} may admit equilibria which do not 
have the form \eqref{eq:oim-equilibria}, however they 
are always unstable or structurally unstable \cite{YC-MKB-NS-ZL:24}, 
so we do not consider them here.

The dynamics~\eqref{eq:oim-dynamics} correspond to the gradient flow 
of the following energy function:
\begin{equation}
    \label{eq:energy}
    E(\theta; \mu) = -\frac{1}{2}\sum_{i=1}^{N}\sum_{j=1}^{N}J_{ij}\cos(\theta_i - \theta_j) + \sum_{i=1}^{N} \mu_i\sin(\theta_i)^2. 
\end{equation}
Note that $\Hc(\sigma) = E(\theta^*(\sigma); \mu) $ for 
all~$\sigma \in \{-1, 1 \}^N$. 
%
The system~\eqref{eq:oim-dynamics} is 
exactly the Kuramoto model on the graph $G$ with 
an additional 
\emph{subharmonic injection term}, $\mu_i \sin(2\theta_i)$. 
We refer to $\mu = (\mu_1,\ldots,\mu_N)$ as the \emph{regularization parameter} 
of \eqref{eq:oim-dynamics}, 
since increasing its value penalizes the energy function for $\theta \in \Torus^N$ 
such that $\theta_i \not\in \{0, \pi\}^N$, 
and ensures that the state converges to equilibria 
with the form~\eqref{eq:oim-equilibria}.

We wish to study the stability of spin configurations 
with respect to the dynamics of the OIM. Because~\eqref{eq:oim-dynamics}
is a gradient system, the stability of~$\theta^*(\sigma)$ 
can be inferred from the Hessian matrix of \eqref{eq:energy}, which is
\begin{equation*}
    \begin{aligned}
        \nabla^2_{\theta_i\theta_j}&E(\theta, \mu) = \\
        &\begin{cases}
            \sum_{k = 1}^{N} J_{ik} \cos(\theta_i - \theta_k) + 2\mu_i \cos(2\theta_i), &i = j, \\
            -J_{ij} \cos(\theta_i - \theta_j), &i \neq j.
        \end{cases}
    \end{aligned}
\end{equation*}
Let $\Hess(\sigma, \mu) = \nabla^2_{\theta\theta}E(\theta^*(\sigma), \mu) \in \real^{N \times N}$ denote the Hessian matrix 
evaluated at the spin configuration $\sigma$. Then,
\begin{equation}
    \label{eq:hessian-at-sigma}
    \Hess_{ij}(\sigma, \mu) = 
    \begin{cases}
        \sum_{k = 1}^{N} J_{ik}\sigma_i\sigma_k + 2\mu_i  & i = j, \\
        -J_{ij}\sigma_i\sigma_j, & i \neq j.
    \end{cases}
\end{equation}
The equilibrium point $\theta^*(\sigma)$  of the dynamical system \eqref{eq:oim-dynamics} is asymptotically stable if $\lambdamin(\Hess(\sigma, \mu)) > 0$. 

Even though \eqref{eq:oim-dynamics} is a gradient system, understanding 
the relationship between global minimizers of the Ising Hamiltonian 
and the stability of the corresponding equilibria is a delicate issue. 
In fact, \emph{$\sigma^*$ being a global optimizer of \eqref{eq:ising_hamiltonian} is neither necessary nor sufficient 
for stability of the corresponding equilibrium $\theta^*(\sigma^*)$}, 
as we show in the following example. 

\begin{example}
\label{ex:minimum-is-not-stable}
Consider a network with $N=3$ nodes:
\[ J = \begin{bmatrix} 0 & -1 & -1 \\ -1 & 0 & -1 \\ -1 & -1 & 0 \end{bmatrix}. \]
The corresponding Ising Hamiltonian has the following set of global minimizers: $\Sigma^* = \{ (1, 1, -1), (1, -1, 1),(1, -1, -1), $ $(-1, 1, 1), (-1, 1, -1), (-1, -1, 1)\}$.
If $\mu_i = 0.1$ for all $i$, it follows that $\lambdamin(\Hess(\sigma^*, \mu)) < 0$ for every $\sigma^* \in \Sigma^*$. Thus, for every minimizer of~\eqref{eq:ising_hamiltonian}, the corresponding 
equilibrium is unstable.
However, if $\mu_i = 10$ for all $i$, then $\lambdamin(\Hess(\sigma, \mu)) > 0$ for all 
$\sigma \in \{-1, 1 \}^N$, implying that the equilibrium corresponding to a spin configuration is asymptotically stable regardless of 
whether it is a minimizer of~\eqref{eq:ising_hamiltonian} or not. 
\QEDA
\end{example}


This example motivates several challenges related to the stability 
analysis of the network system \eqref{eq:oim-dynamics}. First, the
relationship between the stability of 
a spin configuration~$\theta^*(\sigma)$ and the corresponding value of the 
Ising Hamiltonian $\Hc(\sigma)$ is not straightforward. For instance, 
although trajectories of system 
\eqref{eq:oim-dynamics} tend to reduce the value of the energy 
function $E(\theta; \mu)$, this does not mean they converge 
to global minimizers of the Ising
Hamiltonian. 
Second, the stability of a spin configuration 
is sensitive to the choice of the regularization parameter $\mu$. If $\mu$
is too small, the OIM fails to stabilize any of the globally optimal 
spin configurations. On the other hand, if $\mu$ is too large, the OIM
also stabilizes spin configurations that are not optimal. Thus, in order to design
OIMs that are practically useful for solving combinatorial optimization problems,
we need to develop strategies for choosing the regularization parameter so that the trajectories converge to globally optimal solutions with high likelihood. 
We approach both challenges, as summarized next. 

\begin{problem}
\label{prob:main-problem}
Given the dynamical system \eqref{eq:oim-dynamics}:
\begin{enumerate}
    \item Identify the relationship between the stability of a configuration  
    $\sigma \in \{-1, 1 \}^N$,
    the Ising Hamiltonian $\Hc(\sigma)$, and the regularization parameter $\mu$. 
    \item Quantify the impact of heterogeneous choices of regularization parameters
    on the convergence of OIMs to global minimizers. 
\end{enumerate}
\end{problem}

%
%

\section{Signed Graphs of the Ising Model}
\label{sec.signedgraph}
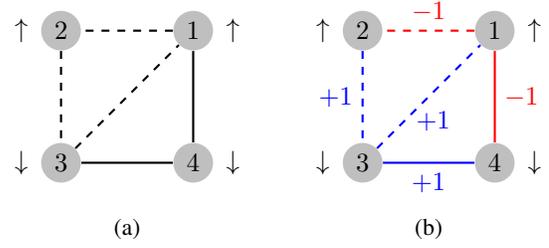
\begin{figure}
    \centering
    \begin{subfigure}{0.45\linewidth}
        \centering
        \begin{tikzpicture}[node distance=15pt]
            \tikzstyle{spinsite}=[circle,fill=black!25,minimum size=15pt,inner sep=0pt, node distance=50pt]
    
            \node (n1) at (0, 0) [spinsite] {$1$};
            \node[spinsite, left of=n1] (n2) {$2$};
            \node[spinsite, below of=n2] (n3) {$3$};
            \node[spinsite, right of=n3] (n4) {$4$};
    
            \draw[thick, dashed] (n1) -- (n2) node[midway, above] {\phantom{$+1$}};
            \draw[thick, dashed] (n2) -- (n3) node[midway, left] {\phantom{$+1$}};
            \draw[thick, dashed] (n3) -- (n1) node[midway, below] {\phantom{$+1$}};
            \draw[thick] (n1) -- (n4) node[midway, right] {\phantom{$+1$}};
            \draw[thick] (n3) -- (n4) node[midway, below] {\phantom{$+1$}};
    
            \node[right of=n1] {$\uparrow$};
            \node[left of=n2] {$\uparrow$};
            \node[left of=n3] {$\downarrow$};
            \node[right of=n4] {$\downarrow$};
        \end{tikzpicture}
        \caption{}
    \end{subfigure}
    \begin{subfigure}{0.45\linewidth}
        \centering
        \begin{tikzpicture}[node distance=15pt]
            \tikzstyle{spinsite}=[circle,fill=black!25,minimum size=15pt,inner sep=0pt, node distance=50pt]
    
            \node (n1) at (0, 0) [spinsite] {$1$};
            \node[spinsite, left of=n1] (n2) {$2$};
            \node[spinsite, below of=n2] (n3) {$3$};
            \node[spinsite, right of=n3] (n4) {$4$};
    
            \draw[thick, dashed, color=red] (n1) -- (n2) node[midway, above] {$-1$};
            \draw[thick, dashed, color=blue] (n2) -- (n3) node[midway, left] {$+1$};
            \draw[thick, dashed, color=blue] (n3) -- (n1) node[midway, below=0.4ex, xshift=0.4ex] {$+1$};
            \draw[thick, color=red] (n1) -- (n4) node[midway, right] {$-1$};
            \draw[thick, color=blue] (n3) -- (n4) node[midway, below] {$+1$};
    
            \node[right of=n1] {$\uparrow$};
            \node[left of=n2] {$\uparrow$};
            \node[left of=n3] {$\downarrow$};
            \node[right of=n4] {$\downarrow$};
        \end{tikzpicture}
        \caption{}
    \end{subfigure}
    \caption{Signed graph construction. (\textbf{a}) Graph of the Ising model and corresponding spin configuration $\sigma$, where ferromagnetic and antiferromagnectic edges are denoted by solid and dashed lines, respectively. (\textbf{b}) Signed graph $G_{\rm s}(\sigma)$, where blue and red lines denote positive and negative edges, respectively.}
    \label{fig:signed-graph}
\end{figure}

In this section, we formulate the problem of analyzing the 
stability of \eqref{eq:oim-dynamics} 
in graph-theoretic language. To this end, we construct a signed graph 
for each spin configuration and relate the Ising Hamiltonian 
and the stability properties of the OIM to the 
spectral properties of the graph Laplacian. 

For each spin configuration $\sigma$, we define the signed graph~$G_{\rm s}(\sigma)$ 
whose signed adjacency matrix is 
given by 
\begin{equation}
\label{eq:signed-adjacency}
    A_{ij}(\sigma) = J_{ij}\sigma_i\sigma_j.
\end{equation}
Following our definition, the edge $(i, j)$ of $G_{\rm s}(\sigma)$ is
positive if either $\sigma_i = \sigma_j$ and $(i, j)$ is ferromagnetic,
or if $\sigma_i \neq \sigma_j$ and~$(i, j)$ is antiferromagnetic, i.e. the local 
interaction energy $-J_{ij}\sigma_i\sigma_j$ between the nodes is minimized. Otherwise, 
the edge is negative. 
The construction of the signed graph is depicted in Figure~\ref{fig:signed-graph}. 
We say that the OIM~\eqref{eq:oim-dynamics} is
\emph{frustrated} if 
$G_{\rm s}(\sigma)$ has at least one negative edge for
every spin configuration $\sigma$, 
otherwise the OIM is \emph{frustration-free}.

The signed graph $G_{\rm s}(\sigma)$ is \emph{structurally balanced} if 
there exists a partition of the set of nodes $\mathcal V=\{1, 2, \dots, N \}$ 
into two sets, $\Vc_1$ and $\Vc_2$,
such that edges within $\Vc_1$ or $\Vc_2$ are positive, and 
edges between $\Vc_1$ and $\Vc_2$ are negative. 
The Laplacian matrix of the graph~$G_{\rm s}(\sigma)$ is $L(\sigma) = D(\sigma) - A(
\sigma)$, where $D(\sigma)=\operatorname{diag}(d_1(\sigma),\ldots,d_N(\sigma))$
is the diagonal matrix formed by the node degrees $d_i(\sigma) = \sum_k A_{ik}(\sigma)$. Unlike the case of unsigned 
graphs, the Laplacian matrix of a signed graph
is not necessarily positive semidefinite.

 By construction, the Laplacian matrix is given by 
 \begin{equation}
     \label{eq:graph-laplacian}
     L_{ij}(\sigma) = \begin{cases}
     \sum_{k=1}^{N}J_{ik}\sigma_i\sigma_k \qquad &i=j, \\
     -J_{ij}\sigma_i\sigma_j \qquad &i\neq j.
 \end{cases}
\end{equation}
It follows that the Ising Hamiltonian is equivalent to
%
\begin{equation}
    \label{eq:hamiltonian-laplacian}
    \begin{aligned}
            \Hc(\sigma) = -\frac{1}{2}\sum_{i=1}^{N}L_{ii}(\sigma) =
    -\frac{1}{2}\trace(L(\sigma)),
    \end{aligned}
\end{equation}
\noindent
and, comparing \eqref{eq:hessian-at-sigma} and \eqref{eq:graph-laplacian}, we see that the Hessian matrix of the energy function is 
\begin{equation}
    \label{eq:hessian-laplacian}
    \Hess(\sigma, \mu) = L(\sigma) + 2M,
\end{equation}
\noindent
where $M = \operatorname{diag}(\mu_1, \dots, \mu_N)$ is a diagonal matrix formed by the regularization parameters $\mu_i$. 
%
Using this signed graph construction, we establish the following stability 
result for frustration-free OIMs. 
\begin{theorem}\longthmtitle{Stability of frustration-free OIMs}
    \label{thm:stability-without-frustration}
    Assume that the OIM is frustration-free.
    Then,
    \begin{enumerate}
        \item If $\sigma^*$ is a global minimizer of~\eqref{eq:ising_hamiltonian}, then 
        $\theta^*(\sigma^*)$ is 
        asymptotically stable for 
        any $\mu \in \real^N_{>0}$.
        \item There exists $\mu^* > 0$ such that for all $\mu \in \real^{N}_{>0}$ where $0 < \mu_i < \mu^*$ for all $i$, any spin configuration 
        $\sigma$ that is not a global minimizer of~\eqref{eq:ising_hamiltonian} is unstable. 
    \end{enumerate} 
\end{theorem}

\begin{proof}
    We begin by showing that if the Ising model is frustration-free, 
    $\sigma^*$ is a global minimizer if and only if every edge in $G_{\rm s}(\sigma^*)$ is 
    positive. This is because if $\sigma^* \in \{-1, 1 \}^N$ 
    is such that every edge of $G_{\rm s}(\sigma^*)$
    is positive, then for all edges~$(i,j)$, 
    we have $J_{ij}\sigma^*_i\sigma^*_j = |J_{ij}|$. 
    Thus the Ising Hamiltonian attains 
    the minimum possible value,
    \[ \Hc(\sigma^*) = -\frac{1}{2}\sum_{i=1}^{N}\sum_{j=1}^{N}\abs{J_{ij}}, \]
    and $\sigma^*$ is a global minimizer. On the other 
    hand, if $\sigma$ is such that $G_{\rm s}(\sigma)$ has a negative edge $(i, j)$, then $J_{ij}\sigma_i\sigma_j < 0<J_{ij}\sigma^*_i\sigma^*_j$, so $\Hc(\sigma) > \Hc(\sigma^*)$ and $\sigma$ is not a global minimizer. 
    
    Next, to prove (i), note that if $\sigma^*$ is a global minimizer, then every edge of $G_{\rm s}(\sigma^*)$ is 
    positive, so $L(\sigma^*) \succeq 0$. Thus
    \[ \Hess(\sigma^*, \mu) = L(\sigma^*) + 2\diag(\mu) \succ 0, \]
    and $\theta^*(\sigma^*)$ is asymptotically stable. 

    To show (ii), let $\sigma \in \{-1, 1 \}^N$ be a spin configuration which is not a global minimizer. 
    We let $\sigma^* \in \{-1, 1 \}^N$ be a global minimizer, 
    and partition the set of the nodes of $G_{\rm s}(\sigma)$ into sets $\Vc_1 = \{ 1 \leq i \leq N | \sigma_i = \sigma_i^* \}$ and $\Vc_2 = \{ 1 \leq i \leq N | \sigma_i \neq \sigma_i^* \}$.
    If $(i, j)$
    is an edge where $i, j \in \Vc_1$, then $J_{ij}\sigma_i\sigma_j = J_{ij}\sigma^*_i\sigma^*_j > 0$, and similarly if $i, j \in \Vc_2$, then $J_{ij}\sigma_i\sigma_j = J_{ij}(-\sigma^*_i)(-\sigma^*_j) > 0$. Otherwise, if $i \in \Vc_1$ and $j \in \Vc_2$, then $J_{ij}\sigma_i\sigma_j = J_{ij}\sigma^*_i(-\sigma^*_j) < 0$. 
    Thus 
    $G_{\rm s}(\sigma)$ is structurally balanced, and since $\sigma$ is 
    not a global minimizer, $G_{\rm s}(\sigma)$ 
    has at least one negative edge. By \cite[Corollary IV.4]{DZ-MB:15}, $L(\sigma)$ is indefinite so $\lambdamin(L(\sigma)) < 0$. If 
    we let
    \[ \mu^* = \min \big\{ -\lambdamin(L(\sigma))\;|\; \sigma \in \{-1, 1 \}^N \text{ is suboptimal}  \big\},  \]
    then for any $\mu \in \real^{N}_{>0}$ where $\mu_i < \mu^*$ 
    for all $i$,  we have 
    \[ \begin{aligned}
        \lambdamin(\Hess(\sigma, \mu)) &= \lambdamin\big(L(\sigma) + \diag(\mu_1, \dots, \mu_N)\big) < 0,
    \end{aligned}  \]
    which implies that $\theta^*(\sigma)$ is unstable. 
\end{proof}

As a consequence of 
Theorem~\ref{thm:stability-without-frustration}, 
it is possible to choose 
regularization parameters
in frustration-free OIMs to guarantee convergence 
to spin configurations corresponding to global 
minimizers of the Ising Hamiltonian, and the proof 
provides a precise bound on the value of $\mu$ 
to achieve this. For large networks,
the OIM \eqref{eq:oim-dynamics} is almost 
always frustrated, and the picture is considerably 
more complicated. Since signed graphs corresponding 
to global minimizers may contain negative edges, 
there may be spin configurations for which the associated
signed graph is unbalanced. 
As a result, estimating $\lambdamin(L(\sigma))$ 
is generally difficult, which in turn makes it challenging to 
properly tune $\mu$ in order
to stabilize global minimimzers while leaving suboptimal 
spin configurations unstable. However, as we show in 
the sequel, by turning to statistical analysis we can 
obtain approximate guarantees with high likelihood for 
ensembes of random graphs. 



\section{Ensembles of Networks}
In this section, we use the signed graph formalism 
introduced in Section~\ref{sec.signedgraph} to characterize how each spin configuration $\sigma$ relates both to the Ising energy $\mathcal H(\sigma)$ and the stability of OIMs that are possibly frustrated. We 
focus specifically on networks with only antiferromagnetic edges. 
Our approach relies on statistically characterizing the relationship between these quantities using an ensemble of random signed graphs, which allows us to compute the \emph{conditional expectation} and \emph{conditional variance} of the eigenvalues of the Hessian matrix of spin configurations $\sigma$ at values 
of the Hamiltonian energy where $\mathcal H(\sigma) = h$. 

We begin by specifying an ensemble, $\mathcal{G}_N(p_1, p_2)$, of randomly-generated graphs $G_{\rm s}(\sigma)$, where $N$ is the number of nodes 
in the graph. The matrix $J$ is randomly constructed following 
the Erdős–Rényi (ER) model, where an antiferromagnetic edge $(i, j)$ is added 
with probability $p_1$, and the spin configuration~$\sigma$ is uniformly sampled with probability~$p_2$. 
Thus, for every $i\neq j$,
\begin{equation}
J_{ij} =
\begin{cases}
0, & \text{w. probability } (1 - p_1), \\
-1, & \text{w. probability } p_1,
\end{cases}
\end{equation}
and
\begin{equation}
\sigma_i =
\begin{cases}
+1, & \text{w. probability } p_2, \\
-1, & \text{w. probability } (1 - p_2).
\end{cases}
\end{equation}
Following~\eqref{eq:signed-adjacency},
the signed graph $G_{\rm s}(\sigma)$ drawn from ensemble $\mathcal G_N(p_1, p_2)$ is given by
\[
\begin{aligned}
A_{ij}(\sigma) 
&= \begin{cases}
0, & \text{w. probability } (1 - p_1), \\
-1, & \text{w. probability } p_1p_2^2 + p_1(1-p_2)^2, \\
+1, & \text{w. probability } 2p_1p_2(1-p_2). 
\end{cases}
\end{aligned}
\]

\noindent
%
The entries of the signed adjacency matrix $A(\sigma)$ are discrete random variables, and the $k$-th moment of 
each entry can be calculated as follows.  
%
%
\begin{lemma}\longthmtitle{Moments of signed adjacency matrix}
    \label{lem:adjacency-moments}
    Let $G_{\rm s}(\sigma)$ be a graph sampled from the 
    ensemble $\Gc_N(p_1, p_2)$, and 
    $A(\sigma)$ be its adjacency matrix. Then, the $k$th 
    moment is
    \begin{equation}
        \label{eq:adjacency-moments}
        \mathbb{E}[A^k_{ij}(\sigma)] = \begin{cases}
        -p_1(2p_2 - 1)^2 \qquad &\text{$k$ odd} \\ 
        p_1 \qquad &\text{$k$ even}
    \end{cases} 
    \end{equation}
\end{lemma}

\begin{proof}
    Observe that, for any $k \in \N$,
    \[ \mathbb{E}[A_{ij}^k] = (-1)^k (p_1p_2^2 + p_1(1-p_2)^2) + 2p_1p_2(1-p_2), \]
    which simplifies to \eqref{eq:adjacency-moments}. 
\end{proof}


\subsection{Moments of the Ising Hamiltonian}
\label{sec.moments.hamiltonian}
We now calculate the expected value and the variance 
of the Ising Hamiltonian
for a signed graph sampled from the ensemble $\mathcal G_N(p_1, p_2)$. 

\begin{lemma}\longthmtitle{Moments of Ising Hamiltonian}
    \label{lem:hamiltonian-moments}
    The 
    expected value of the Ising Hamiltonian 
    for graphs $G_{\rm s}(\sigma)$ sampled 
    from the ensemble $\Gc_N(p_1, p_2)$
    is 
    \[ \mathbb{E}[\Hc(\sigma)] = -\binom{N}{2}\mathbb{E}[A_{ij}(\sigma)] \]
    and its variance is
    \begin{align}
    \notag
        \operatorname{Var}&[\Hc(\sigma)] 
        = \binom{N}{2}\mathbb{E}[A_{ij}(\sigma)^2] + 6\binom{N}{3}\mathbb{E}[J_{ij}]^2\mathbb{E}[\sigma_i]^2 \\
        \label{eq:ising-variance}
        &+ 6\binom{N}{4}\mathbb{E}[J_{ij}]^2\mathbb{E}[\sigma_i]^4 - \binom{N}{2}^2\mathbb{E}[A_{ij}(\sigma)]^2.
    \end{align}
    \noindent
    In the special case where $p_2=\frac{1}{2}$, we have that $\mathbb{E}[\mathcal H(\sigma)] = 0$ and 
        $\operatorname{Var}[\mathcal H(\sigma)] = \frac{1}{2}N(N-1)p_1$.
\end{lemma}
\vspace{1.0ex}

\begin{proof}
    We begin by computing the expectation 
    of the Ising Hamiltonian. 
    Observe that 
    \[ \begin{aligned}
         \mathbb{E}[\Hc(\sigma)] &= -\frac{1}{2}\sum_{i=1}^{N}\sum_{j=1}^{N}\mathbb{E}[J_{ij}\sigma_i\sigma_j] = -\binom{N}{2}\mathbb{E}[A_{ij}(\sigma)].
    \end{aligned} \]
    where the first equality in the previous expression follows by linearity of the expectation operator, and the second 
    equality follows by the fact that 
    there are $\binom{N}{2}$ possible edges in an~$N$
    node network. 

    To compute the variance of the Ising Hamiltonian, we begin by computing its second 
    moment. Observe that
    \begin{equation}
        \label{eq:ising-second-moment}
        \begin{aligned}
            \mathbb{E}&[\Hc(\sigma)^2] = \mathbb{E}\bigg[ \Big( \sum_{(i, j) \in \Ec} A_{ij}\sigma_i\sigma_j  \Big)^2 \bigg] \\
            &=\mathbb{E}\bigg[\sum_{(i, j) \in \Ec} (A_{ij}(\sigma))^2 + \sum_{(i, j) \neq (k, \ell)}A_{ij}(\sigma)A_{k\ell}(\sigma) \bigg],
        \end{aligned}
    \end{equation}
    where $\Ec = \{(i, j) | 1 \leq i\leq N,\, j>i \}$ is the set of all possible edges in an $N$ node network. The first term in~\eqref{eq:ising-second-moment}~is
    \begin{equation}
        \label{eq:sum-of-squares}
        \mathbb{E}\bigg[\sum_{(i, j) \in \Ec} (A_{ij}(\sigma))^2 \bigg] = \binom{N}{2}\mathbb{E}[A_{ij}(\sigma)^2]. 
    \end{equation}
    To compute the second term we need to sum over pairs of distinct edges $(i, j)$ and  $(k, \ell)$. There are two possible cases:
    \begin{enumerate}[leftmargin=15pt]
        \item \emph{Connected edges:} The first case is
        where $(i, j)$ and~$(k, \ell)$ share a common node, i.e.~$j = \ell$. Then, 
        \[ \begin{aligned}
             \mathbb{E}[A_{ij}(\sigma)A_{k\ell}(\sigma)] &= \mathbb{E}[J_{ij}J_{kj}\sigma_i\sigma_k] = \mathbb{E}[J_{ij}]^2\mathbb{E}[\sigma_i]^2,
        \end{aligned} \]
        where the second equality follows 
        since $J_{ij}$, $J_{kj}$, $\sigma_i$, 
        and~$\sigma_k$ are independent random variables. In a network 
        with $N$ nodes, there are $6\binom{N}{3}=N(N-1)(N-2)$ possible pairs of 
        this type. 

        \item \emph{Disconnected edges:} The second case
        is where~$(i, j)$ and~$(k, \ell)$ do note share a common node. Then, 
        \[ \begin{aligned}
             \mathbb{E}[A_{ij}(\sigma)A_{k\ell}(\sigma)] &= \mathbb{E}[J_{ij}J_{k\ell}\sigma_i\sigma_j\sigma_k\sigma_\ell] = \mathbb{E}[J_{ij}]^2\mathbb{E}[\sigma_i]^4,
        \end{aligned} \]
         where the second equality follows 
        since $J_{ij}$, $J_{k\ell}$, $\sigma_i$, $\sigma_j$, $\sigma_k$, 
        and~$\sigma_\ell$ are independent random variables. In a network with $N$ nodes, there are $6\binom{N}{4}=\frac{1}{4}N(N-1)(N-2)(N-3)$
        pairs of edges of this type. 
    \end{enumerate}
    It follows that 
    \begin{equation}
        \label{eq:cross-terms}
        \begin{aligned}
            \mathbb{E}\bigg[\sum_{(i, j) \neq (k, \ell)}A_{ij}(\sigma)A_{k\ell}(\sigma) \bigg] &= 6\binom{N}{3}\mathbb{E}[J_{ij}]^2\mathbb{E}[\sigma_i]^2 \\
            &+ 6\binom{N}{4}\mathbb{E}[J_{ij}]^2\mathbb{E}[\sigma_i]^4.
        \end{aligned}
    \end{equation}
    Finally, we combine \eqref{eq:sum-of-squares} and \eqref{eq:cross-terms}, and use the fact that $\operatorname{Var}[\Hc(\sigma)] = \mathbb{E}[\Hc(\sigma)^2] - \mathbb{E}[\Hc(\sigma)]^2$ to obtain~\eqref{eq:ising-variance}. 
 \end{proof}


\subsection{Moments of the Hessian Spectrum}
\label{sec.moments.hessian}
We now examine the effects of the regularization parameters on the spectrum 
of the Hessian matrix corresponding to a 
signed graph $G_{\rm s}(\sigma)$ drawn from the ensemble $\Gc_N(p_1, p_2)$. 
Here, we are interested in the impact of \textit{disordered} choices for the regularization parameter across individual oscillators, by sampling $\mu_i$ 
from some pre-specified probability distribution. We have 
the following result which characterizes the expected value and variance 
of the eigenvalues of $H(\sigma, \mu)$. 

\begin{lemma}\longthmtitle{Moments of Hessian Spectrum}
    \label{lem:hessian-moments}
    Let $G_{\rm s}(\sigma)$ be a random graph drawn from the ensemble
    $\Gc_N(p_1, p_2)$. If $\lambda$ is an eigenvalue drawn 
    uniformly from the spectrum of the Hessian matrix $H(\sigma, \mu)$, then the expected value of $\lambda$ is
    \begin{equation}
    \label{eq:expected-eig}
        \mathbb{E}[\lambda] = (N - 1)\mathbb{E}[A_{ij}(\sigma)] + 2\mathbb{E}[\mu_i]
    \end{equation}
    and the variance of $\lambda$ is 
    \begin{equation}
        \label{eq:var-eig}
        \begin{aligned}
         \operatorname{Var}[\lambda] &= (N - 1)(N - 2)\mathbb{E}[J_{ij}]^2\mathbb{E}[\sigma_i]^2 + 2(N - 1)\mathbb{E}[J_{ij}^2] \\
        &+ 2(N - 1)\mathbb{E}[\mu_i]\mathbb{E}[A_{ij}(\sigma)] + (N-1)^2\mathbb{E}[A_{ij}(\sigma)]^2 \\
        &+4\operatorname{Var}[\mu_i].
    \end{aligned} 
    \end{equation}
    In the special case where $p_2=\frac{1}{2}$ and $\mu_i \sim \mathcal{U}[a, b]$, we have $\mathbb{E}[\lambda] = (a+b)$ and 
        $\operatorname{Var}[\lambda] = 2(N - 1)p_1 + \frac{1}{3}(b - a)^2$.
\end{lemma}

\begin{proof}
    The expected value of $\lambda$ can 
    be computed using the trace of the Hessian matrix, 
    as follows. 
    \[
        \begin{aligned}
            \mathbb{E}[\lambda] &= \frac{1}{N}\mathbb{E}\left[ \operatorname{tr}\left(H(\sigma, \mu)\right) \right] = \frac{1}{N}\mathbb{E}[\operatorname{tr}(L(\sigma) + 2M)] \\
            &= \frac{1}{N}\mathbb{E}[\operatorname{tr}(L(\sigma))] + 2\mathbb{E}[\mu_i].
        \end{aligned}
    \]
    Using that $\Hc(\sigma) = -\frac{1}{2}\trace(L(\sigma))$, we obtain~\eqref{eq:expected-eig}.  

    We now compute the variance $\operatorname{Var}[\lambda] = \mathbb{E}[\lambda^2] - \mathbb{E}[\lambda]^2$. Note
    that $\mathbb{E}[\lambda^2] = \frac{1}{N}\mathbb{E}[\operatorname{tr}(H(\sigma, \mu)^2)]$, where 
    \[ H(\sigma, \mu)^2 = L(\sigma)^2 + 2ML(\sigma) + 2L(\sigma)M + 4M^2,\]
    and $M = \diag(\mu)$. We compute the trace of each of the terms in the previous expression. 
    The trace of $L(\sigma)^2$ is
    \[ \trace(L(\sigma)^2) = \sum_{i=1}^{N}\sum_{j\neq k}J_{ij}J_{jk}\sigma_i\sigma_{k} + 2\sum_{i=1}^{N}\sum_{j=1}^{N}J_{ij}^2,   \]
    so it follows that 
    \begin{equation}
    \label{eq:exp-L-sq}
        \begin{aligned}
        \mathbb{E}[\trace(L(\sigma)^2)] = &N(N-1)(N-2)\mathbb{E}[J_{ij}]^2\mathbb{E}[\sigma_i]^2\\ &+ 2N(N - 1)\mathbb{E}[J_{ij}^2]. 
    \end{aligned}
    \end{equation}
    Next, observe that 
    \[ \begin{aligned}
        \trace(L(\sigma)M) = \trace(ML(\sigma)) = \sum_{i=1}^{N}\sum_{j=1}^{N}\mu_iA_{ij}(\sigma).
    \end{aligned}  \]
    Since $\mu_i$ is independent of $A_{ij}(\sigma)$, it follows that 
    \begin{equation}
        \label{eq:exp-LM}
        \mathbb{E}[\trace(L(\sigma)M)] = N(N - 1)\mathbb{E}[\mu_i]\mathbb{E}[A_{ij}(\sigma)].
    \end{equation}
    Finally, we have $\mathbb{E}[\trace(M^2)] = N\mathbb{E}[\mu_i^2]$. Combining the previous expressions, we obtain~\eqref{eq:var-eig}. 
    The expressions for $p_2 = \frac{1}{2}$ and $\mu_i\sim\mathcal U[a,b]$ then follow trivially by noting that $\mathbb{E}[\mu_i] = \frac{1}{2}(a+b)$ and $\operatorname{Var}[\mu_i]=\frac{1}{12}(b-a)^2$.
\end{proof}

\subsection{Conditional Moments of the Hessian Spectrum}
\label{sec.moments.conditional}
Building on the results in Section \ref{sec.moments.hamiltonian} and 
Section \ref{sec.moments.hessian}, we can now elucidate the relationship 
between the value of the Ising Hamiltonian and the eigenvalues of the 
Hessian matrix for signed graphs drawn from the ensemble $\Gc_N(p_1, p_2)$. To do this, 
we compute the conditional expectation and estimate the 
variance of the eigenvalues of $H(\sigma, \mu)$ given that 
the spin configuration is at energy level $\Hc(\sigma) = h$. This allows 
us to understand how varying the energy level affects the distribution
of eigenvalues of the Hessian matrix, and in particular, conclude that lower energy 
states are more likely to be stable. 
We thus first establish the following result. 

\begin{theorem}\longthmtitle{Conditional Moment of Hessian Spectrum}
    \label{lem:conditiona-moments}
    Consider the ensemble $\Gc_N(p_1, p_2)$ where $p_2=\frac{1}{2}$, and 
    suppose that $\mu_i \sim \Uc[a, b]$. The conditional expected value of 
    the eigenvalues of $H(\sigma, \mu)$ at energy level $\Hc(\sigma, \mu) = h$~is
    \[ \mathbb{E}[\lambda \mid \mathcal H(\sigma)=h] = - \frac{2}{N}h + (a+b). \]
\end{theorem}

\begin{proof}
The conditional expectation can be calculated~as
\begin{equation}
\begin{aligned}
\mathbb{E}[\lambda \mid& \mathcal H(\sigma)=h] = \frac{1}{N}\mathbb{E}[\operatorname{tr}(H(\sigma, \mu)) \mid \mathcal H(\sigma) = h] \\
&= \frac{1}{N}\mathbb E\left[\operatorname{tr}(L(\sigma)) + 2\trace({M}) \,\Big|\, \mathcal H(\sigma) = h\right] \\
&= \frac{1}{N}\Big(\mathbb E[-2\mathcal H(\sigma) \mid \mathcal H(\sigma)= h] + 2N\mathbb E[\mu_i]\Big) \\
&= - \frac{2}{N}h + (a+b).
\end{aligned}
\label{eq.condexp.lambda}
\end{equation}
\end{proof}

We now present a conjecture that the conditional variance 
of the eigenvalues can be estimated 
by a quadratic function of~$h$ for ensembles 
that tend to generate sparse graphs, i.e. when~$p_1 \ll 1$. 
Although we do not have an analytic expression for the constant $c$ 
in~\eqref{eq:conditional-var-final}, 
it can be estimated numerically. 
As we show in Section~\ref{sec:numerical-results}, 
this theoretical prediction agrees well with the  
empirically computed conditional variance 
in numerical experiments. 

\begin{conjecture}
    Consider the ensemble $\Gc_N(p_1, p_2)$, 
    where~$p_1 \ll 1$ and $p_2=\frac{1}{2}$, 
    and suppose that $\mu_i \sim \Uc[a, b]$. 
    There exists a unique constant $c \in \real$, depending only on~$N$, $p_1$, $a$
    and $b$, where the conditional variance can be approximated 
    by the following quadratic function
    \begin{equation}
        \label{eq:conditional-var-final}
        \begin{aligned}
        \operatorname{Var}[\lambda \mid \Hc(\sigma) = h] \approx &\left(\frac{c}{N}- \frac{4}{N^2}\right)h^2 + \frac{(b-a)^2}{3}  \\
        & + 2(N - 1)p_1.
    \end{aligned}
    \end{equation}
\label{conjecture}
\end{conjecture}

\smallskip
\begin{proofsketch}
The conditional variance is
\begin{equation}
\label{eq.condvar.definition}
\begin{aligned}
\operatorname{Var}[\lambda \mid \mathcal H=h] 
%
&= \mathbb{E}[\lambda^2 \mid \mathcal H=h] - \mathbb{E}[\lambda \mid \mathcal H=h]^2.
\end{aligned}
\end{equation}
Here, the conditional expectation $\mathbb{E}[\lambda^2 | \mathcal H = h]$ is
\begin{equation}
    \label{eq:cond-exp-lambda-sq}
    \begin{aligned}
    \mathbb{E}[&\lambda^2 \mid \Hc = h] \\
    &= \frac{1}{N}\mathbb{E}[\trace(L^2 + 2LM + 2ML + 4M^2) \mid \Hc = h] \\
    &= \frac{1}{N}\mathbb{E}[T(\sigma, \mu) \mid \Hc = h] + 4\mathbb{E}[\mu^2_i]. 
\end{aligned}
\end{equation}
where $T(\sigma, \mu) = \trace(L^2(\sigma) + 2ML(\sigma) + 2L(\sigma)M)$. 
Let $f(h)$ be the unique quadratic function 
that approximates $\mathbb{E}[T(\sigma) \mid \Hc = h]$
with minimum mean squared error. 
It can be shown 
that there exists a constant $c$ such that
\[ \begin{aligned}
    f(h) = &\mathbb{E}[T] + \frac{\operatorname{Cov}(T, \Hc)}{\operatorname{Var}(\Hc)}(h - \mathbb{E}[\Hc]) + c(h - \mathbb{E}[\Hc])^2,
\end{aligned} \]
where $c$ depends only on the parameters that define 
the ensemble $\Gc_N(p_1, p_2)$ and the distribution 
$\Uc[a, b]$.
Because $p_2=\frac{1}{2}$, we have $\mathbb{E}[\Hc(\sigma)] = 0$. 
By \eqref{eq:exp-L-sq} and \eqref{eq:exp-LM}, we have
\[ \begin{aligned}
    \mathbb{E}[T(\sigma)] &= \mathbb{E}[\trace(L(\sigma)^2)] + 4\mathbb{E}[\trace(ML(\sigma))] \\
    &=2N(N - 1)p_1. 
\end{aligned} \]
Next, observe that 
\[ \begin{aligned}
    \operatorname{Cov}[T, \Hc] &= \mathbb{E}[T(\sigma, \mu)\Hc(\sigma)] \\
    &= \mathbb{E}[\trace(L(\sigma)^2)\Hc(\sigma)] + 4\mathbb{E}[\trace(ML(\sigma))\Hc(\sigma)].
\end{aligned}  \]
To simplify the expression for the covariance, 
we note that when $p_1 \ll 1$, i.e. the sampled graphs tend to be
sparse, 
$\mathbb{E}[\trace(L(\sigma)^2)\Hc(\sigma)] \approx 0$ since 
the expectation can be expressed as the sum of terms proportional to $\mathbb{E}[J_{ij}^2]$, which is small when $p_1$ is small. By expanding 
$\trace(ML(\sigma))\Hc(\sigma)$ and using the fact that $\mu_i$ is independent 
of $\Hc(\sigma)$, we obtain $\mathbb{E}[\trace(ML(\sigma))\Hc(\sigma)] = - 2\mathbb{E}[\mu_i]\mathbb{E}[\Hc(\sigma)^2]$. Thus, 
\begin{equation}
    \label{eq:quadratic-approx}
    f(h) = 2N(N-1)p_1 - 4(a + b)h + ch^2.
\end{equation}
Finally, by inspecting \eqref{eq:cond-exp-lambda-sq}, we see that we 
can approximate $\mathbb{E}[\lambda^2 \mid \Hc = h]$ with the the map $h \mapsto \frac{1}{N}f(h) + 4\mathbb{E}[\mu^2_i]$. By substituting this into~\eqref{eq.condvar.definition}, we get the approximation 
of the conditional variance in~\eqref{eq:conditional-var-final}. 
\end{proofsketch}

\begin{figure*}
    \centering
    \includegraphics[width=0.85\linewidth]{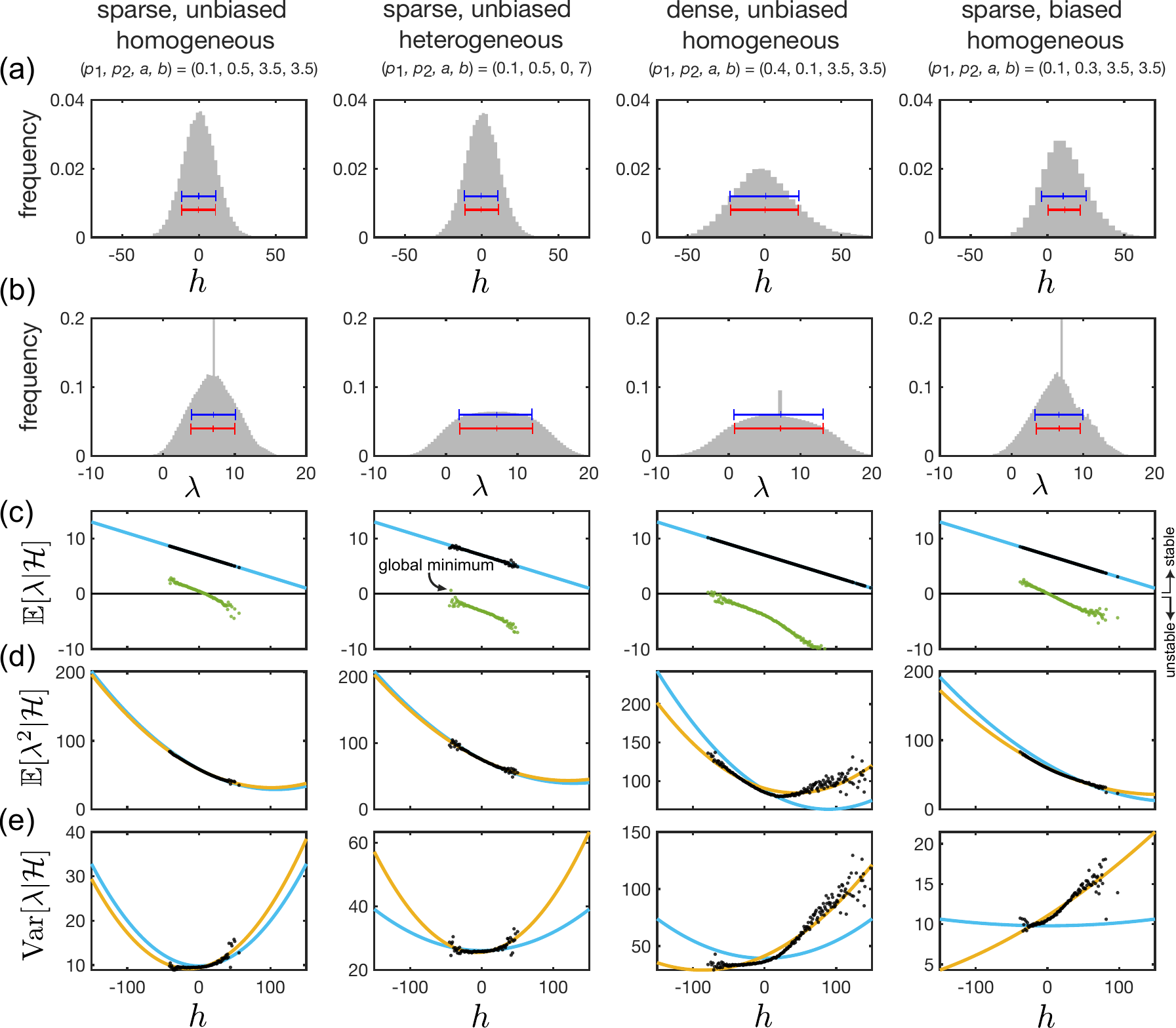}
    \caption{Statistical analysis of the Ising energy and the stability of OIMs.
    \textbf{(a, b)} Distributions of the Ising Hamiltonian energy $\mathcal H$ (a) and Hessian eigenvalues $\lambda$ (b) across $10^5$ realizations of signed graphs from the ensemble $\mathcal G$.
    The horizontal bars indicate the mean $\pm$ standard deviation, with blue and red colors corresponding to theoretical predictions and numerical estimates, respectively.
    \textbf{(c, d, e)} Conditional moments for given values of the Ising Hamiltonian energy. The theoretical predictions are shown by blue lines, while black data points represent numerical estimates. In panel c, green points represent numerical estimates of the conditional expectation of the \textit{smallest} Hessian eigenvalue (i.e., $\mathbb{E}[\lambda_{\rm min} | \mathcal H]$).
    In panels d and e, the yellow line shows a second-order polynomial fit for comparison.
    In all plots, the statistical results were obtained for graphs of $50$ nodes, with the probabilities $p_1$ and $p_2$, and the heterogeneity interval $[a,b]$ reported in the corresponding columns.
    }
    \label{fig.moments}
\end{figure*}

\section{Design Principles for Ising Machines}
\label{sec:numerical-results}


In this section, we demonstrate our theoretic results and their 
implications on the design of OIMs using numerical simulations, and
depict our results in Figure~\ref{fig.moments}. In particular, 
we demonstrate that the conditional variance of the eigenvalues 
of the Hessian increases at extreme values of the Ising energy
when we introduce heterogeneity in the values of the regularization 
parameters, increasing the likelihood that globally optimal spin 
configurations are stable. We analyze the following scenarios: i) homogeneous versus heterogeneous regularization parameters $\mu_i$, ii) sparse and dense graphs $G$, and iii) unbiased and biased distributions of spin configurations (with $p_2 = 0.5$ and $p_2 = 0.3$, respectively).
Following Lemmas~\ref{lem:hamiltonian-moments} and \ref{lem:hessian-moments}, our analytical predictions for the means and variances show strong agreement with the empirical distributions of  Ising energies and Hessian eigenvalues (Fig.~\ref{fig.moments}a,b). As anticipated by Lemma~\ref{lem:hamiltonian-moments}, Fig.~\ref{fig.moments}a shows that the distributions of Ising energies become broader with increasing $p_1$. Moreover, when $p_2<0.5$, $\mathbb E[A_{ij}]$ is larger, and hence the distribution is shifted rightward. For the unbiased case $p_2=0.5$, Fig.~\ref{fig.moments}b shows that the eigenvalue variance increases with the network connectivity (higher $p_1$) and parameter heterogeneity (wider interval $[a,b]$).

Fig.~\ref{fig.moments}c--e illustrates the theoretical predictions of the conditional moments. The predicted conditional expectation $\mathbb E[\lambda | \mathcal H]$ by Theorem~\ref{lem:conditiona-moments} is precise for all ensembles, confirming that the eigenvalues linearly decrease on average as a function of the Ising energy. Hence, equilibria associated with spin configurations of higher energy are more likely to be stable. 
To compute \eqref{eq:conditional-var-final} for the conditional variance, we use data collected from each network ensemble to fit the corresponding constant $c$. 
Notably, the theoretical prediction for the conditional variance $\operatorname{Var}[\lambda | \mathcal H]$ is also accurate when graphs are sparse and unbiased\textemdash which is expected given the approximations underlying Conjecture~\ref{conjecture}. However, the predictions deviate strongly for biased signed graphs ($p_2 = 0.3$) and dense graphs ($p_1 =0.1$). To investigate this issue, we present a second-order polynomial fitting of the data points in Figs. \ref{fig.moments}d,e. It shows that, for large $p_1$ and $p_2\neq 0.5$, the conditional expectation $\mathbb E[\lambda^2 | \mathcal H]$  deviates from our prediction \eqref{eq:cond-exp-lambda-sq}, which directly impacts the approximation~\eqref{eq:conditional-var-final}. Providing a precise analytical second-order expression for $\mathbb E[\lambda^2 | \mathcal H]$ is challenging but likely to yield accurate predictions, a task that we expect to address in future work. 

Crucially, the conditional eigenvalue variance is higher at extreme values of the Ising energy. This finding has direct implications for parameter tuning and the convergence behavior of OIMs. From Eq.~\eqref{eq:conditional-var-final}, it is evident that increasing the variance in the parameter distribution $\mu_i$ (e.g., the width of interval $[a,b]$) leads to a higher conditional variance. This result, combined with the fact that the conditional expectation 
trends downward for higher Ising energies, increases the likelihood that the \textit{smallest} eigenvalue $\lambda_{\rm min}(\Hess(\sigma,\mu))$ associated with the global minimum is considerably larger than those corresponding to suboptimal solutions. This effect is illustrated in Fig.~\ref{fig.moments}e, which shows that the spectral gap between globally optimal and suboptimal equilibria is larger for heterogeneous parameter choices, despite the average parameter being identical to the homogeneous case. Consequently, for a particular network ensemble, one can design the parameter distribution such that equilibria associated with the global minima are stable while the remaining equilibria (associated with suboptimal solutions) are left unstable. 

\section{Conclusion}

This paper presents a theoretical and statistical analysis of the stability properties of OIMs through the lens of signed graph theory. By relating the Ising Hamiltonian to the Laplacian spectrum, we establish a connection between the spin configurations of the Ising optimization problem and the stability of the corresponding equilibria in OIMs. For frustration-free networks, we provide formal guarantees of convergence, which can be used to properly tune the regularization parameter. In such cases, it is possible to choose a homogeneous regularization value such that the global optimum spin configuration is stable, while all other suboptimal configurations are unstable. In the presence of frustration, we analyze random network ensembles and demonstrate that low-energy spin configurations are statistically more likely to to be stable. Finally, we show that introducing heterogeneity in the regularization parameters can enhance convergence toward global minimizers. While we focused on uniform parameter distributions, our framework extends naturally to arbitrary distributions.

Our results align with recent studies emphasizing the beneficial role of disorder in optimization and synchronization. In the context of Ising machines, related approaches have introduced stochastic perturbations\textemdash such as chaotic noise \cite{lee2025noise} and probabilistic simulated annealing \cite{raimondo2025high}\textemdash to dynamically adapt the system parameters. In synchronization theory, parameter heterogeneity has also been shown to enhance synchronization stability \cite{zhang2021random}. In future work, we will focus on characterizing the moments of extremal eigenvalues, e.g., $\lambda_{\rm min}(\Hess (\sigma,\mu))$, in order to develop dynamic tuning strategies for the regularization parameter. This line of research is expected to translate the mathematical formalism and insights presented here into practical design guidelines for OIMs and related Ising-inspired hardware.

\bibliographystyle{ieeetr}



\end{document}